\documentclass[reqno,11pt]{amsart}
\usepackage{graphicx,amsfonts,amssymb,amsmath,amsthm,amscd}
\usepackage{hyperref,color}
\hypersetup{colorlinks,linkcolor=magenta,citecolor=blue,filecolor=black,urlcolor=blue}
\usepackage[body={6.5in, 8.5in},left=1.5in,right=1.5in]{geometry}
\usepackage[all]{xy}

\theoremstyle{plain} %--default

\newcounter{wow}
\newtheorem{theorem}    [wow]{Theorem}
\newtheorem{lemma}      [wow]{Lemma}
\newtheorem{corollary}  [wow]{Corollary}
\newtheorem{proposition}[wow]{Proposition}

\newtheorem{question}   [wow]{Question}

\theoremstyle{definition}
\newtheorem{definition} [wow]{Definition}

\newtheorem{example}    [wow]{Example}

\theoremstyle{remark}
\newtheorem{remark}[wow]{Remark}

 \newcommand{\sect}{%
   \removelastskip%
   \vskip.5\baselineskip%
   \refstepcounter{subsection}%
   \noindent%
   {\bf (\thesubsection)}%
   \hspace{3pt}}

\def\Sym{\operatorname{Sym}}
\def\SO{\operatorname{SO}}

\def\id{\operatorname{id}}

\def\Mon{\operatorname{Mon}}

\def\Hom{\operatorname{Hom}}

\def\ext{\operatorname{ext}}
\def\Stab{\operatorname{Stab}}
\def\Gr{\operatorname{Gr}}
\def\Ext{\operatorname{Ext}}
\def\divo{\operatorname{div}}
\def\ch{\operatorname{ch}}
\def\Br{\operatorname{Br}}
\def\Pic{\operatorname{Pic}}
\def\Hilb{\operatorname{Hilb}}
\def\NS{\operatorname{NS}}
\def\chern{\operatorname{c}}
\def\ST{\operatorname{ST}}
\def\PGL{\operatorname{PGL}}
\newcommand{\Td}{\operatorname{Td}}

\renewcommand{\P}{\mathbb{P}}
\newcommand{\Q}{\mathbb{Q}}
\newcommand{\Z}{\mathbb{Z}}
\newcommand{\C}{\mathbb{C}}
\newcommand{\R}{\mathbb{R}}
\newcommand{\G}{\mathbb{G}}

\renewcommand{\O}{\mathcal{O}}

\renewcommand{\phi}{\varphi}
\renewcommand{\tilde}[1]{\widetilde{#1}}
\renewcommand{\bar}[1]{\overline{#1}}

\newcommand{\into}{\rightarrow}

\newcommand{\K}[1]{$\mathrm{K}3^{[#1]}$}

\newcommand{\Kthree}{K3\;}

\newcommand{\class}[1]{\mathbf{#1}}

\newcommand{\HalgX}{\tilde{H}_\mathrm{alg}(X,\Z)}

\newcommand{\real}[1]{\ldots}
 \renewcommand{\chern}[1]{\mathrm{c}_{#1}}

 \title[Lagrangian planes in holomorphic symplectic varieties]{A Classification of Lagrangian planes in holomorphic symplectic varieties}
 \date{\today}
 \author{Benjamin Bakker}
 \address{B. Bakker:
 Institut f\"ur Mathematik, Humboldt-Universit\"at zu Berlin.
 }
 \email{benjamin.bakker@math.hu-berlin.de}


\begin{document}

% \dedicatory{}

 \begin{abstract}
Classically, an indecomposable class $R$ in the cone of effective curves on a \Kthree surface $X$ is representable by a smooth rational curve if and only if $R^2=-2$.  We prove a higher-dimensional generalization conjectured by Hassett and Tschinkel:  for a
holomorphic symplectic variety $M$ deformation equivalent to a Hilbert
scheme of $n$ points on a \Kthree surface, an extremal curve class $R\in H_2(M,\Z)$ in the Mori cone is the line in a Lagrangian $n$-plane $\P^n\subset M$ if and only if certain intersection-theoretic criteria are met.  In particular, any such class satisfies $(R,R)=-\frac{n+3}{2}$ and the primitive such classes are all contained in a single monodromy orbit.
 \end{abstract}

\maketitle

%\tableofcontents
\section*{Statement of Results}
  Let $M$ be an (irreducible) holomorphic symplectic
variety---that is, a smooth simply connected projective variety admitting a unique (up to scalars) everywhere nondegenerate holomorphic two-form.  $M$ comes equipped with
a quadratic form $(\cdot,\cdot)$ on $H^2(M,\Z)$ called the
Beauville--Bogomolov form; it is primitive, integral, nondegenerate, and
deformation-invariant of signature $(3,b_2(M)-3)$. A \Kthree surface $X$, for example, is
holomorphic symplectic, and the Beauville--Bogomolov form in this case is simply the
intersection pairing.  The Hilbert scheme of $n$ points on $X$, $M=X^{[n]}$, is holomorphic symplectic as well, and the Beauville--Bogomolov form yields an orthogonal decompositon
\[H^2(M,\Z)=H^2(X,\Z)\oplus\Z\delta\]where $H^2(X,\Z)$ is isometrically embedded via pullback along
the Hilbert--Chow map $X^{[n]}\into X^{(n)}$, and $(\delta,\delta)=2-2n$, where $2\delta$ is the divisor of
nonreduced subschemes.  More generally, any proper moduli space $M(\class{v})$ of
stable sheaves on $X$ of Mukai vector $\class{v}$ with
$\class{v}^2=2n-2$ is a holomorphic symplectic variety deformation-equivalent to a Hilbert scheme of $n$ points on a
\Kthree surface.  We say in this case $M$ is  ``of \Kthree type," or sometimes ``of \K{n} type" if we want to specify the dimension.

%. BOGO DECOMP

Classically, much of the geometry of a projective \Kthree surface is encoded in
the intersection pairing on the N\'eron-Severi group $\NS(X)$.  If $h$ is an ample divisor, then the closed cone $\overline{NE}_1(X)$ of effective curves (also called the Mori cone), for instance, is the closure of the cone generated by
\[\{ R\in \NS(X)\mid h.R>0\textrm{ and }R^2\geq -2\}\] 
A primitive curve class $R$ of nonpositive self-intersection generating an extremal ray of $\overline{NE}_1(X)$ is dual to a face of the nef cone whose generic divisor induces either:  a contraction of a smooth rational curve $R$ to an ordinary double point when $R^2=-2$; or an elliptic fibration $X\into\P^1$, when $R^2=0$. 

A program to analogously understand the birational geometry of $M$ purely in terms
of the intersection theory of the Beauville--Bogolomov form was initiated by Hassett
and Tschinkel in
\cite{ratlcurves} and proven for fourfolds in \cite{moving}.  Great strides toward fleshing out this program in higher dimensions have been made recently  by using
Bridgeland stability conditions to analyze moduli spaces $M(\class{v})$ and then
deforming to arbitrary \Kthree type varieties.  Indeed, due to
work of Bayer--Macr\`i \cite{BM}, Bayer--Hassett--Tschinkel
\cite{BHT}, and Mongardi \cite{Mongardi} there is now a complete
description of the nef, movable, and Mori cones of $M$ (see Section \ref{nefcone} below).  In particular,
\begin{theorem}[Proposition 2 of \cite{BHT} or Corollary 2.4 of
  \cite{Mongardi}]\label{extreme}  Let $M$ be a holomorphic symplectic variety of
  \Kthree type and dimension $2n$.  If $R\in H_2(M,\Z)$ is the primitive generator of an extremal ray of the Mori cone, then $(R,R)\geq-\frac{n+3}{2}$.
\end{theorem}
\noindent Here we have used the embedding $H^2(M,\Z)\subset H_2(M,\Z)$ induced by the Beauville--Bogomolov form, and the resulting extension of $(\cdot,\cdot)$ to a rational form on $H_2(M,\Z)$.  

The next step in the Hassett--Tschinkel program is to classify the
geometry of extremal contractions in terms of the intersection theory of their contracted curves.  The exceptional loci of such contractions generically look like a fibration of $k$-dimensional projective spaces\footnote{If $k=1$, an ADE configuration of rational curves can occur in the generic fiber.} over a $(2n-2k)$-dimensional holomorphic symplectic variety, contracting via the projection (see for example \cite{namikawa,projspaces}).  In particular, Lagrangian planes contract to points.  

Our main result is to provide a numerical classification of curve classes $R\in H_2(M,\Z)$ that sweep out a Lagrangian plane $\P\subset M$, thus proving a conjecture of Hassett and Tschinkel \cite{extremal}.  In particular, we have the

%LAGRANGIAN FIBRATION

%4folds and work of wierzba
%
%$n\leq 4$
%
%Shepherd-Barron

\begin{theorem}\label{main}
Let $M$ be a holomorphic symplectic variety of \Kthree type and dimension $2n$, and suppose $R\in
H_2(M,\Z)$ is the class of a line in a Lagrangian $n$-plane $\P^n\subset M$.  Then $R$ satisfies $(R,R)=-\frac{n+3}{2}$ and $2R\in H^2(M,\Z)$.  
\end{theorem}
In Theorem \ref{better} we classify such curve classes (in particular nonprimitive ones) in terms of
Markman's extended weight 2 Hodge structure.  In general the numerical criteria of Theorem \ref{main} are likely not sufficient (see Example \ref{egg}), but for primitive extremal classes they are:

\begin{theorem}\label{iff}  With $M$ as above, a primitive class $R\in H_2(M,\Z)$ generating an extremal ray of the Mori cone is the line in a Lagrangian plane if and only if $(R,R)=-\frac{n+3}{2}$ and $2R\in H^2(M,\Z)$. 
\end{theorem}

In view of Theorem \ref{extreme} and Theorem \ref{iff}, the extremal contractions of
Lagrangian planes are singled out as being the ``most extreme,'' in
the sense that the class of
the line achieves the minimal square of the generator of an extremal
ray in the Mori cone.  They are the higher-dimensional analog of $-2$-curves on \Kthree surfaces.  On the other side of the extremal contraction spectrum, Markman \cite{lagfib} (in the general case) and Bayer--Macr\`i \cite{BM} (in the case of Bridgeland moduli spaces) have recently resolved a long-standing conjecture asserting that a nef class $D$ with $(D,D)=0$ induces a fibration $M\into \P^n$ with Lagrangian tori fibers.  As an application of Theorem \ref{main}, we have a necessary condition for the existence of sections of Lagrangian fibrations:

\begin{corollary}  $M$ admits a Lagrangian fibration with a section only if $H_2(M,\Z)\cap H^{n-1,n-1}(M)$ contains the sublattice
\[\begin{pmatrix}
0&1\\
1&-\frac{n+3}{2}
\end{pmatrix}\]  
\end{corollary}

Lagrangian planes are also of interest because twisting by their structure sheaf yields an autoequivalence of the derived category $D^b(M)$ in much the same way that a smooth rational curve on a \Kthree  surface yields a spherical twist (see \cite{huythom}).

%In the course of the proof we will deduce as an easy consequence of \cite{BHT} a result that is interesting in its own right:  a Lagrangian $n$-planes in a \Kthree type variety $M$ is contractible in some birational model of $M$:
%
%\begin{theorem}\label{introcontract}  Let $M$ be as in Theorem \ref{main}, and suppose $\P^n\subset M$ is a Lagrangian $n$-plane.  There is a birational model $M'$ of $M$ and a birational morphism $f:M\dashrightarrow M'$ whose exceptional locus does not meet $\P^n$ such that $f(\P^n)$ is contractible in $M'$.
%\end{theorem}
%
%Theorem \ref{introcontract} answers a question raised in \cite{projspaces} in the \Kthree type case.  Note that it may be necessary to move to a distinct birational model $M'$ before the Lagrangian plane becomes contractible, and a contractible Lagrangian plane is not necessarily floppable.

Theorem \ref{main} was demonstrated for
varieties of \K{2} type by Hassett--Tschinkel \cite{moving}, for those of \K{3} type by Harvey--Hassett--Tschinkel \cite{HHT}, and for those of \K{4} type by the author and A. Jorza
\cite{laghyp} using the representation theory of the monodromy
group to exhibit possible classes of lines sweeping out a Lagrangian plane 
as integral points on arithmetic curves.  As a product of the analysis, in all three cases a universal formula for the class $[\P^n]\in H_{2n}(M,\Z)$ in terms of Hodge classes and the class of the line $R$ is determined.  It follows that there is a unique
orbit of the classes of such lines under the Zariski closure of the
monodromy group $\Mon(M)$.  We conclude from Theorem \ref{main} that the same is true without taking the Zariski closure for primitive classes:

\begin{corollary}\label{bettermonodromy} The primitive classes $R\in H_2(M,\Z)$ occurring as the line in a Lagrangian plane belong to a single $\Mon(M)$ orbit.
\end{corollary} 

In general it may not be the case that the class of the line is primitive (see Remark \ref{primrem}).  We expect our method to also allow for an intersection-theoretic classification of Lagrangian planes in holomorphic symplectic manifolds deformation-equivalent to generalized Kummer varieties, using the work of \cite{yoshkummer}.  It was conjectured by Hassett and Tschinkel that an analog of Theorem \ref{main} is true but with $(R,R)=-\frac{n+1}{2}$, and the necessity of the numerical conditions was established in the case of fourfolds \cite{kummer}.  Corollary
\ref{bettermonodromy}, however, is not salvageable in this case; it is expected (and proven for $n=2$ \cite{kummer}) that there will always be multiple monodromy orbits.
\subsection*{Outline}  In Section \ref{bayermacri} we summarize the theory of Bridgeland stability conditions on \Kthree surfaces, and the Bayer--Macr\`i description of the nef cones of Bridgeland moduli spaces.  In Section \ref{summary} we define and give some examples of Lagrangian planes and Grassmannians.  In Section \ref{bridgeland} we prove classify Lagrangian planes for Bridgeland moduli spaces, and discuss other contractible Lagrangian subvarieties.  In Section \ref{general} we extend the classification to arbitrary \Kthree type varieties by deformation. 
\subsection*{Acknowledgements}  The author would like to thank Brendan Hassett, Eyal Markman, and Yuri Tschinkel for many helpful discussions.  Particular thanks goes to Arend Bayer for correcting an error in a previous version of this note and enlightening discussions related to Section \ref{general}.  The author was supported by NSF fellowship DMS-1103982 during the time this work was completed.

\section{Bayer--Macr\`i Description of the nef Cone}\label{bayermacri}
We very briefly summarize the basic theory of Bridgeland stability
conditions and the Bayer--Macr\`i description of the nef cone of
Bridgeland moduli spaces on \Kthree surfaces, or at least as much as
we will need.  For background on the first, see Bridgeland's
original paper \cite{bridgelandk3} or Macr\`i's survey in
\cite[Appendix D]{FMbook}.  For the second, Bayer and
Macr\`i develop the theory for general stability conditions in
\cite{BM1} and apply it to the case of \Kthree surface in \cite{BM}, and our summary is mainly taken from their treatment.
\subsection{}\label{1}

Let $X$ be a smooth projective variety, $D^b(X)$ its bounded
derived category of coherent sheaves.  A stability
condition $\sigma=(Z,\mathcal{P})$ consists of a (group) homomorphism
$Z:K(X)\into \C$ and full extension-closed abelian subcategories
$\mathcal{P}(\phi)\subset D^b(X)$ for each $\phi\in \R$ such that
\begin{itemize}
\item Any $0\neq E\in \mathcal{P}(\phi)$ has $Z(E)\in\R_{>0}e^{i\pi\phi}$;
\item $\mathcal{P}(\phi+1)=\mathcal{P}(\phi)[1]$ for all $\phi\in\R$;
\item $\Hom(\mathcal{P}(\phi),\mathcal{P}(\phi'))=0$ if $\phi>\phi'$;
\item Any $0\neq E\in D^b(X)$ has a Harder-Narasimhan filtration,
  \emph{i.e.} there is a sequence in $D^b(X)$
\[0=E_0\into E_1\into\cdots \into E_{n-1}\into E_n=E\]
whose factors $A_i$, defined as the cones
\[E_{i-1}\into E_i\into A_i\into E_{i-1}[1]\]
satisfy $A_i\in \mathcal{P}(\phi_i)$ with
\[\phi_1>\phi_2>\cdots>\phi_{n-1}>\phi_n\]
\end{itemize}
In this case we denote $\phi^+(E)=\phi_1$ and $\phi^-(E)=\phi_n$, and for any interval $I\subset\R$, $\mathcal{P}(I)\subset D^b(X)$
is defined as the full subcategory of $E\in D^b(X)$ for whom
$[\phi^-(E),\phi^+(E)]\subset I$.  It is easy to show that the
stability condition $\sigma$ induces a $t$-structure on $D^b(X)$ whose
heart is $\mathcal{P}((0,1])$.

\subsection{}\label{2}
The homomorphism $Z$ is called the central charge, and we say that
$0\neq E\in D^b(X)$ has phase $\phi$ if $Z(E)\in\R_{>0}e^{i\pi\phi}$.  The objects of
$\mathcal{P}(\phi)$ are called $\sigma$-semistable, while the simple
objects are called $\sigma$-stable.  Under mild technical assumptions
(fullness, see \cite{bridgeland}),
the categories $\mathcal{P}(\phi)$ are Artinian,
and every $\sigma$-semistable object $E\in \mathcal{P}(\phi)$ has a
Jordan--H\"older filtration in the above sense with $\sigma$-stable
factors of phase $\phi$.  Two $\sigma$-semistable objects $E,E'\in
\mathcal{P}(\phi)$ are $S$-equivalent if their Jordan-H\"older factors
are the same (up to permutations).

\subsection{}\label{3}
We further say that the stability condition $\sigma$ is
numerical if the central charge $Z$ factors through the cokernel
$K_\mathrm{num}(X)$ of
the Chern character $\ch:K(X)\into H^*(X,\Q)$.  Let $\Stab(X)$ be the space of numerical stability conditions.
It has a natural metric topology \cite[Proposition 8.1]{bridgeland}
which makes the map $Z:\Stab(X)\into \Hom(K_{\mathrm{num}}(X),\C)$ a
local homeomorphism.  There is a particularly well-behaved connected
component $\Stab^\dagger(X)\subset\Stab(X)$ containing stability
conditions for which the structure sheaves of points $k(x)$ are
all stable of the same phase.  Numerical stability conditions have been constructed on surfaces (see
\cite{bridgelandk3} for \Kthree surfaces and \cite{bertram} in
general).
\vskip 1em
For the remainder of this section we restrict our attention to \Kthree surfaces $X$.  We
note in passing that we could just as easily work throughout with a
\Kthree surface $X$ twisted by a Brauer class $\alpha\in \Br(X)$ using
the results of \cite{twisted}, and in fact this is extremely useful
(\emph{e.g.} the use of Lemma 6.3 in the proof of Theorem 1.3 in
\cite{BM1}, based on the idea of \cite{MYYY}).

\subsection{}\label{4}
The Mukai lattice $\tilde{H}(X,\Z)$ is
the full cohomology of $X$
\[\tilde{H}(X,\Z)=H^0(X,\Z)\oplus H^2(X,\Z)\oplus H^4(X,\Z)\]
endowed with the (pure) weight 2 Hodge structure determined by
$\tilde{H}^{2,0}=H^{2,0}$.  The Mukai pairing of two vectors
$\class{a}=(r,D,s),\class{b}=(r',D',s')\in \tilde{H}(X,\Z)$ is 
\[(\class{a},\class{b})=D.D'-rs'-r's\]
For any object $E\in D^b(X)$, the Mukai vector of $E$ is 
\[\class{v}(E)=\ch(E)\sqrt{\Td(X)}=(\ch_0(E),\ch_1(E),\ch_2(E)+\ch_0(E))\]
Recall that for a \Kthree surface, $\ch:K_\mathrm{num}(X)\into
\tilde{H}(X,\Z)$ is an integral isomorphism.  Thus, by Grothendieck--Riemann--Roch, for $E,F\in D^b(X)$,
\[\chi(R\Hom(E,F))=-(\class{v}(E),\class{v}(F))\]
The algebraic Mukai lattice is defined to be the integral classes in the $(1,1)$ part of the
Mukai lattice, $\tilde{H}_\mathrm{alg}(X,\Z)=\tilde{H}^{1,1}\cap \tilde{H}(X,\Z)$.  Note that the Chern character endows
$K_\mathrm{num}(X)$ with a natural (pure) weight 2 Hodge structure
such that $\ch:K_\mathrm{num}(X)\xrightarrow{\cong}\tilde{H}(X,\Z)$,
and this is perhaps a more elegant definition of the Mukai lattice.

The central charge $Z$ of any numerical stability condition on $X$ can be
represented as $Z(\cdot)=(\Omega_Z,\cdot)$ for a unique $\Omega_Z\in
\HalgX\otimes\C$; we denote the resulting map by the same letter
$\Omega:\Stab(X)\into \HalgX\otimes\C$.

\subsection{}\label{5}
For any Mukai vector $\class{v}\in \HalgX$, the space $\Stab(X)$ has a wall and chamber structure with
respect to $\class{v}$.  That is, there is a locally finite collection
of codimension 1 submanifolds
called walls such that for $\sigma$ off a wall, the set of $\sigma$-stable
objects $E$ (and
thus also the set of 
$\sigma$-semistable objects) is locally constant.  A connected
component of the complement in $\Stab(X)$ of the union of all walls is
called
an open chamber, and its closure a closed chamber.  We say that a
stability condition $\sigma$ is generic with respect to $\class{v}$ if
it lies in an open chamber.  

\subsection{}\label{6}
By \cite[\S 5]{BM}, every (codimension 1) wall
$\mathcal{W}\subset\Stab(X)$ with respect to $\class{v}$ has an associated saturated signature
(1,1) sublattice $\class{v}\in\mathcal{H}\subset\HalgX$, intrinsically
described as the set of $\class{w}\in \HalgX$ for which $Z(\class{w})$
and $Z(\class{v})$ are $\R$-linearly dependent (\emph{i.e.}
$\Im\frac{Z(\class{w})}{Z(\class{v})}=0$) for all
$\sigma\in\mathcal{W}$.  Not all such hyperplanes
$\mathcal{H}$ arise in this way, but we can always associate to $\mathcal{H}$ the
\emph{potential wall} $\mathcal{W}\subset\Stab(X)$ of stability conditions
$\sigma$ such that $Z(\mathcal{H})$ lies on a real line.  In this case
we say $\class{w}\in\mathcal{H}$ is effective if there is a $\sigma$-semistable object with Mukai vector
$\class{w}$ for a generic $\sigma\in\mathcal{W}$, and we define the
effective cone of the potential wall $\mathcal{C}\subset \mathcal{H}\otimes\R$ to be the real
cone generated by effective classes $\class{w}\in\mathcal{H}$. 

\subsection{}\label{7}
Fixing $\sigma$, for any algebraic
space $S$ we say that an $S$-perfect
$E\in D^b(S\times X)$ is a \emph{flat family} if the derived restriction $E_s:=i_s^*E\in
D^b(X)$ is in the heart $\mathcal{P}((0,1])$ for all closed points $s\in S$.  We say $E$ is a flat
family of $\sigma$-(semi)stable objects of Mukai vector $v$ and phase $\phi$ if $E_s$
is $\sigma$-(semi)stable with $\class{v}(E_s)=\class{v}$ for all
closed points $s\in S$.  By \cite{lieb},\cite{toda}, and \cite{inaba},
\begin{itemize}
\item The stack $\frak{M}_\sigma(\class{v},\phi)$ (resp. $\frak{M}^s_\sigma(\class{v},\phi)$)
of flat families of $\sigma$-semistable (resp. $\sigma$-stable) objects is an Artin stack of
finite type over $\C$ with coarse space $M_\sigma(\class{v},\phi)$ (resp. $M^s_\sigma(\class{v},\phi)$).
\item $\frak{M}^s_\sigma(\class{v},\phi)\subset
  \frak{M}_\sigma(\class{v},\phi)$ is an open substack.
\item $\frak{M}^s_\sigma(\class{v},\phi)$ is a $\G_m$-gerbe over a
  symplectic algebraic space $M^s_\sigma(\class{v},\phi)$.
\item If $\frak{M}^s_\sigma(\class{v},\phi)=
  \frak{M}_\sigma(\class{v},\phi)$ then $M_\sigma(\class{v},\phi)$ is
  proper.
\end{itemize}
The situation is even better for a generic stability condition, by
results of \cite{yoshioka} and \cite{toda}:
\begin{theorem}[Theorem 2.13, \cite{BM}]\label{todaexist}
 Suppose $\sigma$ is generic with repect to $\class{v}$.  If $\class{v}=m\class{v}_0$ for $\class{v}_0$ primitive with
  $\class{v}_0^2\geq -2$ and
  $m>0$, then $M_\sigma(\class{v},\phi)(\C)$ is nonempty if and only
  if $\class{v}_0^2\geq -2$.  Furthermore, 
  \begin{enumerate}
  \renewcommand{\theenumi}{\roman{enumi}}
  \item if $\class{v}_0^2> 0$ then $M_\sigma(\class{v},\phi)$ is of the expected
  dimension $\class{v}^2+2$;
  \item   if $m=1$ then there are no strictly $\sigma$-semistable objects of Mukai vector
$\class{v}$ and $M_\sigma(\class{v},\phi)$ is
  smooth and projective of dimension $\class{v}^2+2$.
  \end{enumerate}
\end{theorem}
Henceforth we will typically drop the phase $\phi$ from the
notation (because it is determined up to shifts by $\class{v}$).
\subsection{}\label{8}
By \ref{5}, the spaces
$M=M_{\tau}(\class{v})$ are canonically identified as $\tau$
varies in an open chamber $\Delta$.  For any
$\sigma\in\Stab^\dagger(X)$ Bayer--Macr\`i \cite[Lemma 3.3]{BM1} construct a divisor class $\ell_\sigma$ on $M$,
which evaluates on any map $C\into \frak{M}_{\tau}(\class{v})$ from
a projective curve $C$ with associated flat family $E\in D^b(C\times
X)$ as 
\[\ell_\sigma.C=-\Im\left(\frac{Z(q_*E)}{Z(\class{v})}\right)\]
where $q_*$ is the derived pushforward along the second projection
$q:C\times X\into X$.  $\ell_\sigma.C=0$ if and only if $E_c$ and
$E_{c'}$ are $S$-equivalent for generic $c,c'\in C$.  We have
\begin{theorem}[Theorem 1.2 of \cite{BM}]  The resulting map
  $\ell:\Stab(X)\into H^2_{\mathrm{alg}}(M,\Z)\otimes\R$ is piecewise analytic.  Further,
\begin{enumerate}
\renewcommand{\theenumi}{\roman{enumi}}
\item The image of $\ell$ is the intersection of the big and movable cones of
  $M$, and the birational model associated to
  $\ell_\sigma$ (lying in an open chamber of the movable cone decomposition) is $M_\sigma(\class{v})$.
\item $\ell$ maps the closed chamber $\bar{\Delta}_\sigma$ containing a generic stability condition $\sigma$ to the nef cone of $M_\sigma(\class{v})$.
\end{enumerate}
\end{theorem}

\subsection{} \label{nefcone}
Suppose $E\in D^b(X\times M)$ is a universal object (though a
quasiuniversal object would suffice).  We define a map
$\theta:\class{v}^\perp\into H^2_\mathrm{alg}(M,\Z)$ via
\[\class{v}^\perp\xrightarrow{-\class{v}^{-1}(\cdot)^\vee}K_\mathrm{num}(X)\xrightarrow{\Phi_E}K_\mathrm{num}(M)\xrightarrow{\det}H^2_\mathrm{alg}(M,\Z)\]
where $\Phi_E(\cdot)=q_*(E\otimes p^*(\cdot))$ is the $K$-theoretic
Fourier-Mukai transform, and $p:X\times M\into X$ is the first projection.  $\theta$ is an isometry, and there is a dual map
$\theta^\vee:\HalgX\into H^\mathrm{alg}_2(M,\Z)$.  By \ref{6} every face of the nef cone of $M$
has the form $\theta(\mathcal{H}^\perp)$, and by \cite{BM} such an
$\mathcal{H}$ can be taken to contain $\class{a}\in\mathcal{H}$ with
$\class{a}^2\geq -2$ and
$|(\class{a},\class{v})|\leq\frac{\class{v}^2}{2}$.  Similarly, ever face of the movable cone is of the form $\theta(\mathcal{H}^\perp)$ for $\mathcal{H}$ containing $\class{a}$ with either:  $\class{a}^2=-2$ and $(\class{a},\class{v})=0$; or $\class{a}^2=0$ and $(\class{a},\class{v})=1$ or $2$.  These three possibilities correspond to divisorial contractions of Brill--Noether, Hilbert--Chow, and Li--Gieseker--Uhlenbeck type, respectively.

\subsection{}\label{10}(See Section 14 of \cite{BM}.)
Given a hyperbolic lattice $\mathcal{H}\subset\HalgX$ with potential wall $\mathcal{W}$ and effective cone $\mathcal{C}$, there are only finitely many $\class{a}\in\mathcal{C}$ for which $\class{v}-\class{a}\in \mathcal{C}$.  By the results of \cite{BM}, we can therefore usually assume that $\mathcal{W}$ is not totally semistable with respect to any effective class.  
We define a partition $P=[\class{v}=\sum \class{a}_i]$ of $\class{v}$ to be a set of Mukai vectors $\{\class{a}_i\}$ with $\class{v}=\sum\class{a}_i$, and a second partition $P'=[\class{v}=\sum \class{b}_i]$ is a refinement of $P$ if it can be obtained by partitioning each $\class{a}_i$.  With the above hypothesis, there is an associated locally closed stratum $M_P\subset M_\sigma(\class{v})$ of objects whose Jordan-H\"older factors with respect to a generic $\sigma_0\in\mathcal{W}$ have Mukai vectors $\class{a}_i$.  Moreover, a the stratum $M_{P'}$ lies in the closure of a stratum $M_P$ if and only if $P'$ refines $P$. 

% \begin{definition}The positive cone of $\HalgX$ is the closure of the
%   cone in $\HalgX\otimes \R$ generated by $\{\class{a}\in\HalgX|(\class{a},\class{a})>0\}$.
%   The Mukai cone associated to $\class{v}\in\HalgX$ is
%   the closed cone in $\HalgX\otimes\R$  
% \end{definition}   
% \begin{theorem}
% \end{theorem}

% For convenience we
% make the following definition.
% \begin{definition}\label{defn}Let $X$ be a \Kthree surface and choose
%   some ample $\omega$.  The Mukai cone of a $X$ is the
%   real cone in $\HalgX\otimes\R$ generated by
% \[\left\{\class{a}\in \HalgX|\class{a}^2\geq-2,|(\class{a},\class{v})|\leq\frac{\class{v}^2}{2},\mathrm{\;and\;}(e^{i\omega},\class{a})>0\right\}\]
% \end{definition}
% We then have
% \begin{theorem}[Theorem 12.2 in \cite{BM}]  The Mori cone of $M$ is
%   generated by the positive cone and the image of the Mukai cone under $\theta^\vee$.
% \end{theorem}
\section{Lagrangian planes}\label{summary}
%\subsection{Lagrangian planes}
A Lagrangian subspace of a symplectic vector space is a
maximal isotropic subspace; in particular, it is half-dimensional.
Let $M$ be a holomorphic symplectic variety of dimension $2n$.  A \emph{Lagrangian subvariety} is an embedded smooth subvariety $Z\subset M$ whose tangent space at each point is a Lagrangian subspace.  Note that $\Omega_Z\cong N_{Z/M}$, where $N_{Z/M}$ is the normal bundle of $Z$ in $M$.  A \emph{Lagrangian plane}
$\P\subset M$ is a Lagrangian subvariety isomorphic to projective space $\P=\P^n$.

A Lagrangian plane $\P\subset M$ can always be contracted to a point in the analytic category, and in fact in the category of algebraic spaces:  there is an algebraic space $M'$ and a map $f:M\into M'$ such that $f(\P)=p$ is a point and $f:M\smallsetminus \P\xrightarrow{\cong}M'\smallsetminus p$.  It may happen that $M'$ is not projective.  We say that $\P$ is \emph{extremal} if the class $R\in H_2(M,\Z)$ of the line in $\P$ generates an extremal ray of the Mori cone.  Note that the exceptional locus of the associated extremal contraction may strictly contain $\P$.
\begin{example}\label{examples}
\renewcommand{\theenumi}{\roman{enumi}}
\begin{enumerate}
\item The prototypical example of a Lagrangian plane is the zero section $\P\subset \mathbb{A}(\Omega_\P)$ in the total space of the cotangent bundle $\Omega_\P$.  Blowing up $X=\mathbb{A}(\Omega_\P)$ at $\P$, the exceptional fiber is isomorphic to the universal hyperplane in $\P\times\P^\vee$, and can be blown down to the second factor yielding a smooth manifold $X'$ containing $\P^\vee$.  This is called the Mukai flop, and $X$ and $X'$ admit contractions of $\P$ and $\P^\vee$, respectively, to the same analytic space $X_0$.  Any Lagrangian plane $\P\subset M$ is locally analytically isomorphic to the Mukai flop, and can similarly be flopped to a complex manifold $M'$ (again, even an algebraic space if $M$ is a variety), but the resulting manifold $M'$ need not be K\"ahler.
\item Let $X$ be a \Kthree surface containing a smooth rational curve $C\cong \P^1\subset X$.  Let $X\into X'$ be the contraction of $C$ to a double point.  There is a natural embedding $\Sym^n C\cong \P^n\subset X^{[n]}$, and the plane $\P^n$ is contractible via the Hilbert--Chow morphism $S^{[n]}\into \Sym^n X'$, though the subscheme $2\delta$ of nonreduced subschemes is contracted as well.
\item Here is an example due to Namikawa of a Lagrangian plane whose flop is not projective \cite[Example 1.7.ii]{namikawa}.  Let $X\into \P^1$ be a projective elliptic \Kthree surface with two $I_3$ fibers (\emph{i.e.} cycles of three smooth rational curves), one of which is $E_1+E_2+E_3$.  As in the previous example, there are three disjoint Lagrangian planes $E_i^{(2)}\subset X^{[2]}$, and flopping all three yields a nonprojective manifold.
\item If $L$ is an effective line bundle on a \Kthree surface $X$ such that every section of $L$ is reduced and irreducible, then let $\mathcal{C}\subset X\times \P$ be the universal divisor over $\P=\P H^0(L)^\vee$.  The compactified relative Jacobian $\overline\Pic_\P(\mathcal{C})\into \P$ is a moduli space of stable sheaves on $X$, and any section is a Lagrangian plane.  In particular, the structure sheaf gives a section of $\overline\Pic_\P^0(\mathcal{C})$.
\end{enumerate}
\end{example}

The machinery of the previous section (see \ref{nefcone}) allows one to very concretely describe the nef and movable cones of moduli spaces, and we give here an in-depth look at Hilbert schemes $X^{[2]}$ of two points on a \Kthree surface $X$ with Picard rank one.  In this case the classification of birational transformations reduces to two Pell's equations (see \cite[\S 13]{BM}).  

Let $\Pic(X)=\Z h$ with $h$ the ample generator of degree $h^2=2d$.  As described in the introduction, we have an isomorphism
\[H^2(X^{[n]},\Z)=\Z h\oplus\Z\delta\]
and the decomposition is orthogonal with respect to the Beauville--Bogomolov form.  Furthermore, we have $(h,h)=2d$ and $(\delta,\delta)=-2$.  $h$ is represented by the divisor of subschemes one of whose points is supported on a fixed hyperplane section of $X$, and $2\delta$ is the divisor of nonreduced subschemes.  Note that $h$ always generates an extremal ray of the nef (and movable) cone as it induces the Hilbert--Chow morphism $X^{[2]}\into X^{(2)}$ contracting the diagonal.

\begin{example}\label{egg}
\renewcommand{\theenumi}{\roman{enumi}}
\begin{enumerate}
\item For $d=1$, $X$ is a degree two cover $X\into\P^2$ branched over a sextic, and hyperplane sections of $X$ are genus $2$ curves mapping to a line in $\P^2$ via the unique hyperelliptic cover.  The Hilbert scheme $X^{[2]}$ has a Lagrangian plane, the closure of the set of reduced fibers of the map $X\into \P^2$.  The degree $2$ compactified Jacobian $\overline\Pic^2(\mathcal{C})$ of the universal hyperplane section $\mathcal{C}$ also has a Lagrangian plane, the section $\P\subset \overline\Pic^2(\mathcal{C})$ given by restricting the polarization $\O(h)$.  In fact, $\overline\Pic^2(\mathcal{C})$ is the Mukai flop of $X^{[2]}$, and the flop is resolved by the relative Hilbert scheme of two points $\Hilb^2(\mathcal{C})$.  One can show that the movable cone of $X^{[2]}$ is $\langle h,h-\delta\rangle$, which is decomposed into the nef cone $\langle h, 3h-2\delta\rangle$ of $X^{[2]}$ and the image $\langle 3h-2\delta,h-\delta\rangle$ of the nef cone of the flop $\overline{\Pic}^2(\mathcal C)$.  The isotropic divisor $h-\delta$ induces the Langrangian fibration $\overline\Pic^2(\mathcal C)\into \P^{2\vee}$, and $\overline\Pic^2(\mathcal C)$ is the only other birational model of $X^{[2]}$.  The wall between them is generated by the nef class $3h-2\delta$ contracting the Lagrangian plane, and the class of the line \[R=h-\frac{3}{2}\delta\] satisfies $R^2=-\frac{5}{2}$ and $2R\in H^2(X^{[2]},\Z)$.  
\item For $d=11$, the movable cone of $X^{[2]}$ is $\langle h, 10h-33\delta\rangle$.  There are two chambers corresponding to the two birational models:  the nef cone of $X^{[2]}$ is $\langle h, 7h-22\delta\rangle$, and that of the birational model is $\langle 7h-22\delta,10h-33\delta\rangle$.  The wall between them is once again a flop, and the contracted curve \[R_1=h-\frac{7}{2}\delta\] again has $R_1^2=-\frac{5}{2}$ and $2R_1\in H^2(X^{[2]},\Z)$.  Note however that
\[R_2=11h-\frac{73}{2}\delta\]
is in the Mori cone and also has these two properties, but the movable cone does not intersect $R_2^\perp$.
\end{enumerate}
\end{example}

Lagrangian Grassmannians are also of interest; we'll likewise say an embedded Lagrangian Grassmannian $\Gr(k,\ell)\subset M$ in a holomorphic symplectic variety is extremal if the class of the minimal rational curve in $\Gr(k,\ell)$ is extremal in the Mori cone of $M$.  The following example of Hassett and Tschinkel \cite[Remark 3.1]{extremal} shows how Lagrangian Grassmannians naturally arise.
\begin{example}\label{grasseg}Let $X\subset\P^3$ be a general quartic (in particular, one containing no lines), and $M=X^{[4]}$ the Hilbert scheme of four points on $X$.  $M$ contains a Lagrangian Grassmannian:  intersecting with any line $\ell\subset \P^3$ gives a length 4 subscheme, and there is an embedding $\Gr(2,4)\subset M$.  Let $\P=|\O_X(1)|$ and take $\mathcal{C}\subset X\times \P$ to be the universal hyperplane section.  The locus of subschemes supported on a hyperplane section contains this $\Gr(2,4)$, and is the image of a map from the relative Hilbert scheme of $\mathcal{C}$:
\begin{equation}\xymatrix{
F\ar[d]\ar[r]&\Hilb^4_\P(\mathcal{C})\ar[d]\ar[r]&X^{[4]}\\
\P\ar[r]&\overline\Pic^4_\P(\mathcal{C})&
}\label{diagram}\end{equation}
A generic line bundle of degree 4 on a hyperplane section $C$ will have a two dimensional space of sections, so $\Hilb^4_\P(\mathcal{C})$ is generically a $\P^1$-bundle over $\overline\Pic_\P^4(\mathcal{C})$.  The locus where the fiber jumps to a $\P^2$ is the section $\P\into \overline\Pic_\P^4(\mathcal{C})$ obtained by restricting $\O(1)$ to a hyperplane section, and the preimage $F$ of this in $\Hilb^4_\P(\mathcal{C})$ is the flag variety of $\P^3$ parametrizing hyperplanes and lines contained in them.  The leftmost map is one of the forgetful maps $F\into \P$, and the composition of the top arrows is the other one $F\into \Gr(2,4)\subset M$.
\end{example}
Lagrangian subvarieties are rigid in the following sense:\begin{lemma}\label{deform}Let $G\subset M$ be a Lagrangian Grassmannian in a holomorphic symplectic variety.  Then $G$ does not deform as a subscheme.  If $G\cong \P$ is a Lagrangian plane, then no curve $C\subset \P$ deforms out of $\P$.
\end{lemma}
\begin{proof}
For the first statement, since $G$ is Lagrangian, we have $N_{G/M}\cong \Omega_{G}$ and therefore $H^0(N_{G/M})=0$.  For the second, it follows from the Euler sequence
\[0\into \Omega_\P\into \O_P(-1)^{n+1}\into \O_\P\into 0\]
together with the sequence
\[0\into I_{\P/M}\into I_{C/M}\into I_{C/\P}\into0\]
that $\Hom(I_{\P/M},\O_C)=0$, and therefore that the map $\Hom(I_{C/\P},\O_C)\into \Hom(I_{C/M},\O_C)$ is an isomorphism.
\end{proof}

\section{Bridgeland Moduli Spaces}\label{bridgeland}
Let $X$ be a \Kthree surface,
$\class{v}\in \HalgX$ a primitive Mukai vector with $\class{v}^2>0$ and $\sigma$ a generic
stability condition.  We first prove
Theorem \ref{main} for $M=M_\sigma(\class{v})$.  Note that by Verbitsky's Torelli theorem \cite{verbitsky} and Markman's results on the monodromy group for the \Kthree deformation type (see Section \ref{general}), the weight 2 Hodge structure $\tilde H(X,\Z)$ from \ref{4} together with the class $\class{v}\in \HalgX$ determines $M_\sigma(\class{v})$ up to birational equivalence.
\begin{definition}  A \emph{pointed period} $(\tilde \Lambda,\class{v})$ is a (pure) weight 2 polarized Hodge structure on the Mukai lattice $\tilde\Lambda$ with Hodge number $h^{2,0}=1$ together with a primitive algebraic class $\class{v}\in\tilde\Lambda_\mathrm{alg}$.  A \emph{pointed sublattice} is a saturated sublattice $\mathcal{H}\subset\tilde\Lambda_\mathrm{alg}$ containing $\class{v}$.  We will adopt the convention that when a pointed sublattice is specified by its intersection form, the distinguished class will be the first basis vector.
\end{definition}

Roughly speaking, if there is a partition
$\class{v}=\class{a}+\class{b}$ with $\class{a},\class{b}\in \HalgX$, then
$\sigma$-stable objects with Mukai vector $\class{v}$ can be built
as extensions of objects $A,B$ with Mukai vectors $\class{a}$ and
$\class{b}$, and the projectivized extension group $\P=\P\Ext^1(A,B)^\vee$
will map into $M_\sigma(\class{v})$.  The geometry of $\P$
depends on the particulars of the pointed sublattice $\mathcal{H}$ generated by $\class{a}$ and $\class{b}$.  Recall that
$\class{a}\in\HalgX$ is \emph{spherical} if $\class{a}^2=-2$.  An object $A\in D^b(X)$ is \emph{rigid} if $\Ext^1(A,A)=0$, and \emph{spherical} if it is rigid and $\Hom(A,A)=\C\id$.

\begin{definition}\label{ptype}  A pointed sublattice $\mathcal{H}\subset\HalgX$ is a $\P$ \emph{type} sublattice if:
\begin{itemize}
\item[(i)] There is a spherical class $\class{s}\in \mathcal{H}$ such that $|(\class{s},\class{v})|=\frac{\class{v}^2}{2}$.
\item[(ii)] There is no spherical class $\class{s'}\in\mathcal{H}$ with $|(\class{s'},\class{v})|<\frac{\class{v}^2}{2}$.
\end{itemize}
Further we say a $\P$ type sublattice $\mathcal{H}$ is \emph{extremal} with respect to a (generic) stability condition $\sigma$ if $\theta(\mathcal{H}^\perp)$ is a wall of the nef cone of $M_\sigma(\class{v})$.
\end{definition}
Note that $\mathcal{H}$ being extremal with respect to some stability condition is equivalent to $\theta(\mathcal{H}^\perp)$ intersecting the movable cone of each $M_{\sigma}(\class{v})$.
A pointed sub-lattice of the form
\[\begin{pmatrix}
\class{v}^2&\frac{\class{v}^2}{2}\\\frac{\class{v}^2}{2}&-2\\
\end{pmatrix}\]
is automatically of $\P$ type, though not every one is of this form.  For the following lemma, we say that $\class{v}$ is minimal in $\mathcal{H}$ if there is no effective spherical class $\class{s}\in\mathcal{H}$ with $(\class{s},\class{v})<0$.
\begin{lemma}\label{little}Let $\mathcal{H}$ be a $\P$ type sublattice, and let $\mathcal{W}$ be the associated potential wall with generic $\sigma_0\in \mathcal{W}$.  If $\class{v}$ is minimal in $\mathcal{H}$, then there are two $\sigma_0$-stable spherical objects $S,T$ with Mukai vectors $\class{s},\class{t}$ such that $\class{v}=\class{s}+\class{t}$. 
\end{lemma}
\begin{proof} By definition, there is a spherical class $\class{s}\in\mathcal{H}$ with $(\class{s},\class{v})=\frac{\class{v}^2}{2}$, and it is effective. Note that $\class{t}=\class{v}-\class{s}$ satisfies
\[\class{t}^2=\class{v}^2-2(\class{s},\class{v})+(-2)=-2\]
so $\class{t}$ is spherical, and moroever $(\class{t},\class{v})=\frac{\class{v}^2}{2}$ as well. 
By \cite[Proposition 6.3]{BM}, there are exactly two $\sigma_0$-stable objects $S,T$ with Mukai vectors $\class{s}_0,\class{t}_0$, and each Jordan-H\"older factor of an object representing $\class{s}$ is $\class{s}_0$ or $\class{t}_0$.  Thus, $\class{s}=x\class{s}_0+y\class{t}_0$ with $x,y\geq 0$, but since
\[(\class{s},\class{v})=x(\class{s}_0,\class{v})+y(\class{t}_0,\class{v})=\frac{\class{v}^2}{2}\]
we must have either $\class{s}=\class{s}_0$ or $\class{s}=\class{t}_0$ by condition $(\mathrm{ii})$.  Similarly, $\class{t}=\class{t}_0$ or $\class{t}=\class{s}_0$.    
\end{proof}
The utility of Definition \ref{ptype} is hinted at by the following:  
\begin{lemma}\label{reverse}If $\mathcal{H}\subset\HalgX$ is an extremal $\P$ type sublattice, then $M_\sigma(\class{v})$ contains an extremal Lagrangian plane $\P\subset M_\sigma(\class{v})$ for $\sigma$ generic on either side of the wall associated to $\mathcal{H}$.
\end{lemma}
%\begin{proof}
%Choose $\sigma$ to be generic with respect to $\class{a}$ and
%$\class{b}$ as well.  Note that
%\[(\class{a},\class{b})=\frac{1}{2}(\class{v}^2-\class{a}^2-\class{b}^2)=n+1\]
%By Theorem \ref{todaexist}, there will
%exist $\sigma$-stable objects $A,B$ with $\class{a}=\class{v}(A)$ and
%$\class{b}=\class{v}(B)$.  By shifting and rotating $\sigma$, we can
%assume that $A$ and $B$ are in the heart $\mathcal{P}((0,1])$.
%$\class{a},\class{v}$ necessarily span a two dimensional lattice in
%$\HalgX$, so $Z(\class{a})$ and $Z(\class{v})$ are not
%colinear (generically).  Assuming $\phi(A)<\phi(B)$, then
%$\Ext^k(B,A)=0$ for $k\leq 0$ and $k>2$, so $\ext^1(B,A)\geq
%(\class{a},\class{b})=n+1$.  Letting $\P=\P\Ext^1(B,A)^\vee$, the universal extension 
%\[p^*A(1)\into E\into p^*B\into p^*A(1)[1]\]
%defines a flat family $E\in D^b(\P\times
%X)$ with $\class{v}(E_x)=\class{v}$ for all $x\in\P$.  Restricting to $x\in\P$ and
%applying $\Hom(B,\cdot)$ to the above sequence, we see that
%$\Hom(B,E_x)=0$, and therefore by Lemma \ref{extension} $E_x$ is
%$\sigma$-stable.  The resulting map $\P\into
%M_\sigma(\class{v})$ is injective on points.  In fact, the pullback of
%the holomorphic symplectic form to $\P$ must vanish, and $\P$ is thereby
%forced to be smoothly embedded and half-dimensional
%(and in particular $\Ext^2(B,A)=0$).  The case $\phi(A)>\phi(B)$ is similar.
%\end{proof}  
\begin{proof} 
First assuming that $\class{v}$ is minimal in $\mathcal{H}$, by Lemma \ref{little} there exist $\sigma_0$-stable objects $S,T\in \mathcal{P}_0(1)$ of classes $\class{s},\class{t}$ (respectively) for a generic $\sigma_0\in\mathcal{W}$, and $\class{v}=\class{s}+\class{t}$.  We therefore have that
$\Ext^k(T,S)=0$ for $k< 0$ and $k> 2$ as $S,T$ are both in the heart of $\sigma_0$, and further $\Hom(T,S)=\Ext^2(T,S)=0$ by stability, so $\ext^1(T,S)=
(\class{s},\class{t})=n+1$.  If $\sigma$ is a generic stability condition on one side of $\mathcal{W}$, assume $\phi(S)<\phi(T)$ and let $\P=\P\Ext^1(T,S)^\vee$.  Denoting by $p:\P\times X\into \P$ the first projection,
\begin{equation}p^*S(1)\into E\into p^*T\into p^*S(1)[1]\label{sphext}\end{equation}
defines a flat family $E\in D^b(\P\times
X)$ with $\class{v}(E_x)=\class{v}$ for all $x\in\P$.  Restricting to $x\in\P$ and
applying $\Hom(T,\cdot)$ to the above sequence, we see that
$\Hom(T,E_x)=0$, and therefore by the following simple lemma, $E_x$ is
$\sigma$-stable.  
\begin{lemma}[Lemma 6.9 of \cite{BM}]\label{simple}Let $A,B$ be simple objects in an abelian category, and 
\[0\into A^x\into E\into B^{y}\into 0\]
any extension with the property that either:  (i) $x=1$ and $\Hom(B,E)=0$; or (ii) $y=1$ and $\Hom(E,A)=0$.
Then in case (i) every proper quotient of $E$ is isomorphic to $B^z$ for some $z$; in case (ii) every proper subobject of $E$ is isomorphic to $A^z$ for some $z$.
\end{lemma}

Further, it is easy to see that any $\sigma$-stable object $E$ with Jordan-H\"older partition $P=[\class{v}=\class{s}+\class{t}]$ with respect to $\sigma_0$ is of the form \eqref{sphext}, so we have an isomorphism between the stratum $M_P\subset M_\sigma(\class{v})$ and $\P$.  The case $\phi(S)>\phi(T)$ likewise produces a Lagrangian plane on the other side of the wall.

Finally, if $\class{v}$ is not minimal in $\mathcal{H}$, then there is a minimal $\class{v}_0\in\mathcal{H}$ such that $\class{v}$ is obtained from $\class{v}_0$ by successive spherical reflections.  If we let $\ST:D^b(X)\into D^b(X)$ be the composition of the corresponding sequence of spherical twists by $\sigma$-stable spherical objects, for $\sigma$ on one side of the wall, then by the same argument as \cite[Proposition 6.8]{BM}, $\ST$ applied to the family in \eqref{sphext} will be stable on that side of the wall. 

\end{proof}
We can compute the class of the line in the Lagrangian plane of Lemma \ref{reverse}.  Recall from \eqref{nefcone} that for any curve $C\subset M_\sigma(\class{v})$,
\[\theta(\class{w}).C=(\class{w},\class{v}(\Phi_E(\O_C)))\]
Let $\P^1\subset \P$ be a line and
$q:\P^1\times X\into X$ the projection onto the second factor.  The
class $R=[\P^1]\in H_2(M_\sigma(\class{v}))$ is determined by
intersecting with all divisors:  for any $\class{w}\in \class{v}^\perp$,
\[\theta(\class{w}).R=(\class{w},\class{v}(q_*E|_{\P^1\times
  X}))=(\class{w},2\class{s}+\class{t})=(\class{w},\class{s})\]
and thus $R=\theta^\vee(\class{s})$ in the case $\phi(S)<\phi(T)$.  If
$\phi(S)>\phi(T)$ we obtain $R=-\theta^\vee(\class{s})$.

The Lagangians planes constructed as in Lemma \ref{reverse} are clearly extremal.  Indeed, by the classification in \cite[Theorem
5.7]{BM}, passing through the potential wall associated to the hyperbolic
lattice $\mathcal{H}$ will flop the
projective space and change the sign of the class of the line.  

The rest of this section will be devoted to showing that this is in fact the only way such planes arise: 

% For convenience, given $\class{v}\in\HalgX$, we will
%say that $\HalgX$ admits a hyperbolic
%  lattice of \emph{$\P$ type} (with respect to $\class{v}$) if there is a saturated sublattice
%  $\class{v}\in\mathcal{H}\subset\HalgX$ of signature (1,1) such that $\class{v}$
%  has an extremal spherical decomposition
%  $\class{v}=\class{a}+\class{b}$ for $\class{a},\class{b}\in
%  \mathcal{H}$.  Thus, $M_\sigma(\class{v})$ contains a contractible Lagrangian
%  plane whenever $\HalgX$ admits a hyperbolic lattice of $\P$ type,
%  and i

\begin{proposition}\label{moduli}

Let $\class{v}\in \HalgX$
  be primitive with $\class{v}^2>0$ and $\sigma$ a generic stability
  condition with respect to $\class{v}$.  $M_\sigma(\class{v})$
  contains an extremal Lagrangian plane $\P\subset M_\sigma(\class{v})$ if and only if $\HalgX$ admits an extremal $\P$ type sublattice $\mathcal{H}$ with respect to $\sigma$.  Further, in this case the class of the line in $\P$ is $\pm\theta^\vee(\class{s})$, for $\class{s}\in\mathcal{H}$ a spherical class with $(\class{s},\class{v})=\frac{\class{v}^2}{2}$.
\end{proposition}
\begin{proof} 
The reverse direction and the computation of the curve class follow from Lemma \ref{reverse} and the ensuing
discussion, so we just need to demonstrate the necessity of the
lattice condition.

Suppose for some generic stability condition $\sigma$
there is an extremal Lagrangian plane $\P\subset M_\sigma(\class{v})$, and let
$\pi:M_\sigma(\class{v})\into M$ be the contraction.  $\pi$ is realized by crossing some wall
$\mathcal{W}\subset \Stab(M)$; let $\mathcal{H}\subset \HalgX$ be the
associated hyperbolic lattice, with effective cone
$\mathcal{C}\subset \mathcal{H}\otimes\R$.  Then $\pi$ contracts curves parametrizing objects that are $S$-equivalent with respect to $\sigma_0$, and since $\P$ is contracted to a point, a generic point $x\in\P$ has fixed Jordan-H\"older factors $A_i$ with respect to $\sigma_0$.  Let $P=[\class{v}=\sum \class{a}_i]$ be the corresponding partition, where $\class{a}_i=\class{v}(A_i)$, and let $M_P\subset M_\sigma(\class{v})$ be the locally closed subvariety of points with the same Jordan-H\"older decomposition.  
\begin{lemma}  The $A_i$ are all rigid.  In particular, they are spherical.
\end{lemma}
\begin{proof}Obviously if the $A_i$ deformed, then because the dimensions of the extension groups between the $A_i$ locally remain constant (as the $A_i$ locally remain stable with respect to $\sigma_0$), then a curve $C$ contracted by $\pi$ would deform outside of $\P$, contradicting Lemma \ref{deform}.  Thus, the $A_i$ are simple in $\mathcal{P}_0(1)$ and rigid, and therefore spherical.
\end{proof}
%The $A_i$ are $\sigma_0$-stable and therefore simple in $\mathcal{P}_0(1)$, so the extension groups between them have constant dimension as the $A_i$ vary in moduli.  Thus, the map $M_P\into \prod_i M_{\sigma_0}(\class{a}_i) $ is smooth in a neighborhood of $(A_i)$, and because $\P\subset M_\sigma(\class{v})$ does not deform, it follows that each $\class{a}_i$ is spherical.  
As in the proof of Lemma \ref{reverse}, there are exactly two spherical objects $S,T\in\mathcal{P}_0(1)$, so in fact the partition $P$ must be of the form $\class{v}=x\class{s}+y\class{t}$, where $\class{s}=\class{v}(S)$ and $\class{t}=\class{v}(T)$.  Suppose $\phi(S)<\phi(T)$ for $\sigma$.  It follows that an object $E$ associated to a point of $M_P$ is an extension of the form
\begin{equation}S\otimes U^\vee\into E\into T\otimes V\into S\otimes U^\vee[1]\label{ext}\end{equation}
where $U,V$ are vector spaces of dimensions $x,y$, respectively.  Indeed, since neither $S$ nor $T$ admits self-extensions, \eqref{ext} is just the Harder--Narasimhan filtration on the other side of the wall\footnote{Therefore, we can retrospectively realize that $U^\vee=\Hom(S,E)$ and $V^\vee=\Hom(E,T)$.}.  Such an extension is stable only if the rightmost map $U\otimes V\into\Ext^1(T,S)$ from \eqref{ext}
satisfies
\begin{enumerate}
\item[(i)] $U\into \Hom(V,\Ext^1(T,S))$ is injective, and
\item[(ii)] $V\into \Hom(U,\Ext^1(T,S))$ is injective
\end{enumerate}
and therefore we can identify $M_P$ with a subscheme of the space of bidegree $(1,1)$ maps $\P U\times \P V\into \P\Ext^1(B,A)$, up to the action of $\PGL(U)\times \PGL(V)$.  But $\P\subset M_P$ and 
\[\dim M_P\leq xy\left((\class{s},\class{t})-xy\right)=\frac{\class{v}^2}{2}+x^2+y^2-x^2y^2
=n+(x^2-1)(y^2-1)\]
In order for this to be half-dimensional, we need either $x=1$ or $y=1$.  In this case, by Lemma \ref{simple} we have $M_P=\Gr(x,(\class{s},\class{t}))$ (for $y=1$, and likewise if $x=1$), and in order for $\P=M_P$, we need $x=y=1$.

%FINE, all stable
%Not totally semistable.  appendix\footnote{Jordan'H\"older only defined up to}
\end{proof}
This provides an easy verification that Lagrangian planes contract to isolated singularities, since the decomposition $\class{v}=\class{s}+\class{t}$ cannot be refined:
\begin{corollary}  If $\P\subset M_\sigma(\class{v})$ is an extremal Lagrangian plane, then $\P$ is a connected component of the exceptional locus of the associated extremal contraction $\pi:M_\sigma(\class{v})\into M$.
\end{corollary}
The proof of Proposition \ref{moduli} begs the same classification question for Lagrangian Grassmannians, and a similar argument shows that they arise as the strata corresponding to partitions of the form $P=[\class{v}=\class{s}+k\class{t}]$, for spherical $\class{s},\class{t}$.
%\begin{proposition}\label{grass}Let $\class{v}\in \HalgX$
%  be primitive with $\class{v}^2>0$ and $\sigma$ a generic stability
%  condition with respect to $\class{v}$.  $M_\sigma(\class{v})$
%  contains an extremal Lagrangian Grassmannian $\Gr(k,\ell)\subset
%  M_\sigma(\class{v})$ if and only if there is an extremal partition $\class{v}=k\class{a}+\class{b}$.
%  %, in which case the class of the line $\P^1\subset\P$ is $\pm\theta^\vee(\class{a})$.
%\end{proposition}
%\begin{proof}The only change of note is that in the proof of Lemma \ref{spherical}, we just get that either $i=1$ or $j=1$.
%\end{proof}
As remarked above, however, it is not the case that an extremal Grassmanian $G$ is contracted to an isolated singularity, because the partition
\begin{equation}\class{v}=\class{s}+\class{t}+\cdots+\class{t}\label{part}\end{equation}
is a common refinement of the partitions $\class{v}=(\class{s}+m\class{t})+(k-m)\class{t}$ for all $0\leq m<k$, and since $(\class{s}+m\class{t})^2\geq 0$, these Jordan-H\"older factors deform.  In fact, $G$ will lie in the closure of each of these strata, since spherical objects have no self extensions and therefore \eqref{part} and $\class{v}=\class{s}+k\class{t}$ have the same associated strata.  There will thus always be a rational curve in $G$ which sweeps out a larger exceptional locus (as in Example \ref{grasseg}). 

\begin{example}\label{grasstwo}Here we revisit Example \ref{grasseg} in the above language.  $X^{[4]}$ is the moduli space $M_\sigma(\class{v})$ for $\class{v}=(1,0,-3)$, parametrizing ideal sheaves $I_Z$ of length 4 subschemes $Z\subset X$, for $\sigma$ in some chamber $\mathcal{C}$ of the stability manifold.  The Grassmannian $\Gr(2,4)\subset X^{[4]}$ arises as the ideal sheaves that are complete intersections of hyperplane sections of $X$:
\[0\into \O_X(-2)\into \O_X(-1)^2\into I_Z\into 0\]
and can therefore be thought of as extensions
\[\O_X(-1)^2\into I_Z\into \O_X(-2)[1]\into \O_X(-1)^2[1]\]
Thus, the corresponding partition is $\class{v}=2\class{s}+\class{t}$, for $\class{s}=(1,-H,3)$ and $\class{t}=(-1,2H,-9)$, and $\class{v}=\class{s}+\class{s}+\class{t}$ is a refinement of it.  As in the above, both of these partitions have the same associated stratum, but the second also refines $\class{v}=\class{s}+\class{a}$, where $\class{a}=\class{s}+\class{t}=(0,H,-6)$.  The stratum $M_P$ of $\class{v}=\class{s}+\class{a}$ parametrizes ideal sheaves of the form
\[0\into \O_X(-1)\into I_Z\into F\into 0\] 
with $\class{v}(F)=\class{a}$---that is, with $Z$ lying entirely on a hyperplane section.  $F$ moves in a $2+\class{a}^2=6$ dimensional family isomorphic to $\overline\Pic_\P^4(\mathcal{H})$, and $M_P$ is the image of the complement of the Lagrangian section $\P$ in Example \ref{grasseg}.  Note that $\P$ is the Lagrangian plane corresponding to the partition $\class{a}=\class{s}+\class{t}$, by Lemma \ref{reverse}. 
\end{example}

The phenomenon in Example \ref{grasstwo} is generally true:  a Lagrangian Grassmannian $\Gr(k,\ell)$ in a moduli space $M_\sigma(\class{v})$ always ``comes from" a Lagrangian Grassmannian in a smaller dimensional moduli space $M_\sigma(\class{w})$ with respect to the same stability condition $\sigma$, and the process terminates at a Lagrangian plane.  A general notion of ``stratified Mukai flops" such as these were first studied by Markman \cite{brillnoether}.

\section{Holomorphic symplectic varieties of K3 type}\label{general}  
We now turn to
the general case.  Much of the Hodge-theoretic structure of Bridgeland moduli spaces on \Kthree surfaces is
echoed by arbitrary holomorphic symplectic manifolds of \Kthree type.  For $M$ a \Kthree type manifold of dimension $2n$, Markman \cite[Corollary 9.5]{torelli}
constructs a monodromy invariant  extension of pure weight 2 Hodge
structures (we will blur the notational distinction between the Hodge
structure and the underlying lattice)
\[0\into H^2(M,\Z)\into \tilde\Lambda(M)\into Q(M)\into 0\]
where $\tilde{\Lambda}(M)$ is a (pure) weight 2 Hodge structure on the Mukai lattice $\tilde{\Lambda}$ polarized by the intersection form, and $Q(M)$ is rank 1 of type $(1,1)$.  In the language introduced in the previous section, this yields a pointed period $(\tilde\Lambda(M),\class{v}(M))$ which determines $M$ up to birational equivalence, again by Verbitsky's Torelli theorem \cite{verbitsky}.  The subgroup of the
oriented isometry group
$O^+(\tilde{\Lambda})$ preserving the embedding $H^2(M,\Z)\into
\tilde{\Lambda}(M)$ is equal to $\Mon^2(M)$, the image of the
restriction map $\Mon(X)\into O(H^2(M,\Z))$, and there is a natural
lift of the monodromy action to $\tilde{\Lambda}(M)$.

In the case of a Bridgeland
moduli space $M=M_\sigma(\class{v})$ of objects on a \Kthree surface
$X$, $\tilde{\Lambda}(M)= \tilde{H}(X,\Z)$ is the pointed period described above, and the embedding $H^2(M,\Z)\into
\tilde{\Lambda}(M)$ is the inverse of the Mukai map
$\theta:\class{v}^\perp\into H^2(M,\Z)$.  In general we will still
denote by $\class{v}(M)$ a primitive generator of
$H^2(M,\Z)^\perp\subset\tilde\Lambda(M)$; we always have
$\class{v}(M)^2=2n-2$.  We will also denote by
$\theta^\vee:\tilde{\Lambda}(M)\into H_2(M,\Z)$ the dual of the
embedding.

From \cite{BHT}, the description of the nef cone of moduli spaces in terms of their pointed periods in (\ref{nefcone}) deforms to all holomorphic symplectic varieties of \Kthree type, and in particular, we have

\begin{theorem}[Theorem 1 of \cite{BHT}]\label{nefconedescript}  Let $(M,h)$ be a polarized holomorphic symplectic variety of \Kthree type.  The Mori cone of $M$ is generated by the positive cone and classes of the form 
\[\left\{\theta^\vee(\class{a})\mid\class{a}\in\tilde\Lambda(M)_{\mathrm{alg}}\textrm{ with }\class{a}^2\geq -2, |(\class{a},\class{v})|\leq \frac{\class{v}^2}{2}, h.\theta^\vee(\class{a})>0\right\}\]

\end{theorem}

Before proving the general case of Theorem \ref{main}, recall that a parallel transport operator is an isometry $\phi:H^2(M,\Z)\into H^2(M',\Z)$ that arises from parallel transport of the local system $R^2f_*\Z$ for some smooth proper family $f:\mathcal{M}\into B$ along a path with endpoint fibers $M$ and $M'$.    Recall also that for an embedding $i:Y\hookrightarrow M$ of a Lagrangian submanifold into a holomorphic symplectic
  manifold, the
  deformations
  of the pair $(M,i)$ are those of $M$ that preserve the sub-Hodge structure $\ker i^*\subset H^*(M,\Z)$,
  and they are unobstructed (see \cite{hodge2} and \cite{hodge1}).  In our case, for $Y=\P$ and $R\in H_2(M_\sigma(\class{v}),\Z)$ the
  class of the line $R$, as long as $R$ remains algebraic in a family the
  plane will deform as well.  

%We have the
%
%\begin{lemma}\label{extremal}
%Let $M$ be a holomorphic symplectic variety of K3 type and dimension $2n$ with a Lagrangian plane $\P\subset M$ and let $R\in H_2(M,\Z)$ be the class of the line.  Then there is a smooth proper family $f:\mathcal{M}\into B$ with a section $R$ of $R^{2n-2}f_*\Z$ which remains algebraic and two bassoonist $b_0,b_1\in B$ such that 
%\end{lemma}
%\begin{proof}
%\end{proof}
%and $R\in H_2(M,\Z)$ an extremal curve class of the form $R=\theta^\vee(\class{a})$ with $\class{a}^2\geq -2$ and $|(\class{a},\class{v})|\leq\frac{\class{v}^2}{2}$.  Then there is a birational model $M'$ of $M$ and a parallel transport operator $\phi:H^2(M,\Z)\into H^2(M',\Z)$ such that $\phi(R)$ is extremal in the Mori cone.
%
%\begin{proof}
%$\phi(R)$ of course has the same numerical properties, so by Theorem \ref{nefconedescript} either $\phi(R)$ or $-\phi(R)$ is effective.  The claim then follows by Corollary 6 of \cite{BHT}, whose proof is essentially the fact that the movable cone is tiled by translates of the nef cones of the birational models of $M$ under parallel transport operators.
%\end{proof}
%
%  In particular, a generic such deformation cannot be projective; otherwise, as $R$ is the only curve class a multiple must be the complete intersection of ample divisors contradicting Lemma \ref{deform}.  Thus, $(R,R)<0$. 

We now prove the first main theorem:
  
    \begin{theorem}\label{better}Let $(M,h)$ be a holomorphic symplectic variety of \Kthree type and dimension $2n$ with a Lagrangian plane $\P\subset M$ and let $R\in H_2(M,\Z)$ be the class of the line.  Then $\tilde\Lambda(M)$ admits a $\P$ type sublattice $\mathcal{H}$ and $R=\theta^\vee(\class{s})$ for a spherical class $\class{s}\in \mathcal{H}$ with $|(\class{s},\class{v})|=\frac{\class{v}^2}{2}$.
  \end{theorem}
  \begin{proof}We know that $R^2<0$, and so by the argument of Proposition 3 of \cite{BHT}, there is a smooth proper family $f:\mathcal{M}\into B$ over an irreducible analytic base specializing to $M$ over some point $0\in B$ such that there is an algebraic section $\rho$ of $R^{2n-2}f_*\Z$ specializing to $R$ over $0$ (the argument of Proposition 3 only uses that $R^2<0$ and the line deforms sideways).  By the above discussion, the Lagrangian plane $\P$ also deforms to the general fiber.  As the periods of moduli spaces are dense in the base of the Kuranishi family of the pair $(M,\P)$, we can find a specialization to a moduli space for which the plane $\P$ does not degenerate and such that $\P$ is extremal.  Transporting the $\P$ type lattice guaranteed by \ref{moduli} then yields the claim.

  \end{proof}

  The existence of a $\P$ type lattice does not only depend on $H^2(M,\Z)$, but we always have a simple necessary criterion for a curve class to be the class of a line in a Lagrangian plane:

\begin{corollary}\label{prim}  Let $(M,h)$ be as above, and let $R$ be the class of a line in a Lagrangian plane.  Then
$(R,R)=-\frac{n+3}{2}$ and $2R\in H^2(M,\Z)$.
\end{corollary}
  \begin{proof}  The statement follows from the following observation:
  \begin{lemma}\label{numbers}
%\begin{enumerate} 
%\renewcommand{\theenumi}{\roman{enumi}}
If $\tilde\Lambda(M)$ admits a spherical $\class{a}\in\tilde\Lambda(M)$ such that $(\class{a},\class{v})=\frac{\class{v}^2}{2}$, then $R=\theta^\vee(\class{a})$ has $(R,R)=-\frac{n+3}{2}$ and order 2 in the discriminant group of $H^2(M,\Z)$.
%\item If $R\in H_2(M,\Z)$ is a class with $(R,R)=-\frac{n+3}{2}$ and order 2 in the discriminant group of $H^2(M,\Z)$, then $R=\theta^\vee(\class{a})$ for some spherical $\class{a}\in \tilde\Lambda(M)$ with $(\class{a},\class{v})=\frac{\class{v}^2}{2}$.
%\item If in (2) $R$ is primitive, then the saturation $\mathcal{H}$ of $\langle \class{a},\class{v}\rangle$ is a $\P$ type sublattice.  
%\end{enumerate}
\end{lemma}
\begin{proof}  
Since $\theta^\vee$ is the
composition of the orthogonal projection onto $\class{v}^\perp$ and
the inclusion $H^2(M,\Z)\into H_2(M,\Z)$ given by the quadratic form, we have
\[(R,R)=\left(\class{a}-\frac{\class{v}}{2}\right)^2=\class{a}^2-(\class{a},\class{v})+\frac{\class{v}^2}{4}=-\frac{n+3}{2}\]
We have $2\class{a}-\class{v}\in \class{v}^\perp$, so $R$ is 2-torsion in the
discriminant group $D$ of $H^2(M,\Z)$, but clearly $R\neq 0$ in $D$, so
$R$ has order 2.  

%For (2), given $R$ with $2R\in H^2(M,\Z)$, take $\class{w}\in
%\class{v}^\perp$ with $\theta^\vee(\class{w})=2R$; then
%$\class{w}+\class{v}$ is 2-divisible in $\tilde{\Lambda}$.  Indeed, the generator of $Q$ lifts to some $\class{e}\in\tilde{\Lambda}$ with
%$\Z\class{e}\oplus H^2(M,\Z)\cong \tilde{\Lambda}$.  Since $\class{v}$
%is primitive, $(\class{e},\class{v})=1$, and
%\[(\class{w}+\class{v},\tilde{\Lambda})=(\class{w},H^2(M,\Z))+(1+(\class{w},\class{e}))\Z\]
%We cannot have $2|(\class{w},\class{e})$, or else $\class{w}$ would be
%2-divisible in $\tilde{\Lambda}$ whereas $2R$ is primitive (in
%$H^2(M,\Z))$, and the first part of the claim follows taking $\class{a}=\frac{\class{w}+\class{v}}{2}$.    
%
%
%For (3), if $\class{s}\in\mathcal{H}$ is a spherical class, then $R'=\theta^\vee(\class{s})=kR$, so
%\[(R',R')=-2-\frac{(\class{s},\class{v})^2}{\class{v}^2}=-k^2\frac{n+3}{2}\leq -\frac{n+3}{2}\]
%and thus $|(\class{s},\class{v})|\geq \frac{\class{v}^2}{2}$.
\end{proof}
\end{proof}
  
  Running the argument of Theorem \ref{better} backwards yields a partial converse:
  
  \begin{theorem} \label{converse} Let $(M,h)$ be as above, and suppose $R\in H_2(M,\Z)$ is a primitive generator of an extremal ray of the Mori cone.  Then $R$ is the class of a line in a Lagrangian plane if and only if $(R,R)=-\frac{n+3}{2}$ and $2R\in H^2(M,\Z)$.
  \end{theorem}
\begin{proof}  The ``only if" part follows from the previous theorem, so we only need to prove the sufficiency of the numerical conditions in this setting.  As $R$ is extremal, we know by Theorem \ref{nefconedescript} that it is a multiple of a class of the form $\theta^\vee(\class{a})$ for $\class{a}\in \mathcal{H}\subset\tilde\Lambda(M)$ as in the Theorem, and it is not hard to see that in fact $\mathcal{H}$ must be $\P$ type and further $R=\theta^\vee(\class{s})$ for a spherical class $\class{s}$ with $|(\class{s},\class{v})|=\frac{\class{v}^2}{2}$.  Again by Proposition 3 and Corollary 6 of \cite{BHT}, there is a smooth proper family along which $R$ remains algebraic specializing to a moduli space $M'$ for which the image $R'\in H_2(M',\Z)$ of $R$ is extremal, and therefore by Theorem \ref{moduli} $R'$ is the class of a line in a Lagrangian plane $\P$ which then deforms to the general fiber of the family, as above.  The primitivity and extremity assumptions on $R$ ensure that $\P$ does not degenerate in $M$. 
\end{proof}

A full converse to Corollary \ref{prim} is not expected without some indecomposability constraint on the curve class $R$---indeed, such a hypothesis is also needed in the case of smooth rational curves on \Kthree surfaces---but the exact condition is at the moment unclear.  If we drop the extremity and primitivity condition in Theorem \ref{converse} and only insist that $R$ comes from a $\P$ type lattice, then the argument carries through except at the last step where we must show that the plane $\P$ does not degenerate in $M$.  For example, the class $R_2$ in Example \ref{egg}(ii) is such a class, and there is even a family keeping $R_2$ algebraic whose associated parallel transport operator sends $R_2$ to $R_1$, but the Lagrangian plane on the generic fiber of this family could easily degenerate.

With Corollary \ref{prim} in mind, Corollary \ref{bettermonodromy} follows by lattice theory.  Let $L$ be the lattice $H^2(M,\Z)$, and
$D(L)=L^\vee/L$ its discriminant group.  Denote by $O(L)$ the isometry
group of $L$ and by $\tilde{O}(L)$ the group of isometries acting trivially
on $D(L)$.  By a result of Eichler \cite[\S
10]{eichler} (see also \cite[Lemma 3.5]{GHS}), the orbit of a
primitive class $\class{a}\in
L$ under the group $\tilde{O}(L)$ is determined by its
square $(\class{a},\class{a})$ and the class of its dual
$\class{a}^\vee=\frac{1}{\divo(\class{a})}(\class{a},\cdot)\in D(L)$
in the discriminant group.  Recall that $\divo(\class{a})$ is defined
by $(\class{a},L)=\divo(\class{a})\Z$.  By \cite[Lemma 9.2]{torelli},
$\tilde{O}(H^2(M,\Z))$ is an index 2 subgroup of $\Mon^2(M)$, and therefore we deduce:
\begin{corollary}\label{mono}  There is a single monodromy orbit containing all
  primitive classes arising as lines in 
 Lagrangian planes embedded in holomorphic symplectic varieties of
 \Kthree type.
\end{corollary}
\begin{remark}  In fact, the same proof shows that the number of monodromy orbits containing the classes of lines in Lagrangian planes is at most equal to the number of square divisors of $(n+3)(n-1)$.
\end{remark}

\begin{remark}\label{primrem}It is not in general true that the class of a line in an extremal Lagrangian plane is primitive.  Indeed, if $\class{v}$ is minimal in a $\P$ type sublattice $\mathcal{H}$, so that $\class{v}=\class{s}+\class{t}$ in the notation of Section 4, this will be the case if and only if the parallelogram with vertices $0,\class{s},\class{t},\class{v}$ contains no interior lattice point---\emph{i.e.} if $\class{s}$ and $\class{t}$ generate $\mathcal{H}$.  Since the effective cone is generated by $\class{s},\class{t}$, any other contracted stratum $M_P$ corresponding to a partition $P=[\class{v}=\sum \class{a}_i]$ must have each $\class{a}_i$ in the interior of the parellelogram with vertices $0,\class{s},\class{t},\class{v}$, so this is in turn equivalent to $\P$ being the only exceptional locus.
\end{remark}

In dimensions $\leq 8$, the method of proving the sufficiency of the numerical criteria in Corollary \ref{prim} in \cite{moving,HHT,laghyp} also provides universal expressions for the class $[\P]\in H^{2n}(M,\Z)$ of a Lagrangian plane in terms of Hodge classes and the dual to the class of the line $\rho=2R\in H^2(M,\Z)$:
\begin{align*}
&[\P^2]=\frac{1}{24}\left(3\rho^2+\chern{2}(M)\right)\\
&[\P^3]=\frac{1}{48}\left(\rho^3+\rho\chern{2}(M)\right)\\
& [\P^4]=\frac{1}{337920}\left(880\rho^4+1760\rho^2\chern{2}(M)-3520\theta^2+4928\theta
 \chern{2}(M)-1408\chern{2}(M)^2\right)\\
\end{align*}
Here $\theta\in \Sym^2 H^2(M,\Z)^*\subset H^4(M,\Q)$ is the class of the Beauville--Bogomolov form.  Given the monodromy invariance in Corollary \ref{mono}, such universal expressions must exist.

\begin{question}  What are the universal polynomials for the class of a Lagrangian plane (with primitive line class) in a holomorphic symplectic variety of \Kthree type in terms of the dual to the line and Hodge classes? 
\end{question}

As is clear from the $n=4$ case, the class of a Lagrangian plane cannot always be expressed purely in terms of chern classes and $\rho$.

\bibliography{biblio}
\bibliographystyle{alpha}
\end{document}